\title{Onsager's Conjecture for the Incompressible Euler Equations in the H\"{o}log Spaces $C^{0,\al}_{\la}(\bar{\Omega})$ }
\author{Hugo Beir\~{a}o da Veiga$^{1,}$ \footnote{Partially supported  by FCT (Portugal) under grant UID/MAT/04561/2019.}\qquad Jiaqi Yang$^{2,}$\footnote{Hugo Beir\~{a}o da Veiga (\texttt{hbeiraodaveiga@gmail.com}) and Jiaqi Yang (\texttt{yjqmath@nwpu.edu.cn})}}
\date{
\small $^1$ Department of Mathematics, Pisa University, Pisa, Italy\\
\small $^2$ School of Mathematics and Statistics,
Northwestern Polytechnical University,
Xi'an, 710129, China}
\documentclass[12pt]{article}
\usepackage{amsfonts}
\usepackage{mathrsfs}
\usepackage{amsmath}
\usepackage{amssymb,extarrows}
\usepackage{multicol}
\usepackage{float}
\usepackage{makeidx}
\usepackage{layout}
\usepackage{array}
\usepackage{a4wide}
\usepackage{boxedminipage}
\usepackage{hyperref}
\usepackage{latexsym}
\usepackage{color}
\usepackage{dsfont}
\usepackage{ulem}
\usepackage[dvips]{graphicx}

\usepackage{amsthm}

\newtheorem{theorem}{Theorem}[section]
\newtheorem{proposition}[theorem]{Proposition}
\newtheorem{lemma}[theorem]{Lemma}

\theoremstyle{remark}
\newtheorem{remark}{Remark}[section]

\theoremstyle{definition}
\newtheorem{definition}[theorem]{Definition}

\numberwithin{equation}{section}

\newcommand{\p}{\partial}
\newcommand{\e}{\epsilon}

\newcommand{\R}{\mathbb{R}}
\newcommand{\N}{\mathbb{N}}

\newcommand{\al}{\alpha}
\newcommand{\f}{\frac}
\newcommand{\n}{\nabla}

\newcommand{\la}{\lambda}

\begin{document}
\maketitle
\abstract{}
In this note we extend a 2018 result of Bardos and Titi \cite{BT} to a new class of functional spaces $C^{0,\al}_{\la}(\bar{\Omega})$. It is shown that weak solutions $\,u\,$ satisfy the energy equality provided that $u\in L^3((0,T);C^{0,\al}_{\la}(\bar{\Omega}))$ with $\al\geq\f13$ and $\la>0$. The result is new for $\,\al=\,\f13\,.$ Actually, a quite stronger result holds. For convenience we start by a similar extension of a 1994 result of Constantin, E, and Titi, \cite{CET}, in the space periodic case. The proofs follow step by step those of the above authors. For the readers convenience, and completeness, proofs are presented in a quite complete form.
\vskip 0.4 true cm
\noindent \textbf{Mathematics Subject Classification:} 35Q31, 76B03.

\vspace{0.2cm}

\vspace{0.2cm}
\section{Introduction}
In this note, we are concerned with the Onsager's conjecture of incompressible Euler equations in a bounded domain $\Omega\subset\R^n$ with $C^2$ boundary
\begin{equation}\label{eq:E}
\begin{cases}
\p_tu+u\cdot\n u+\n p=0\,,&\text{in $\Omega\times(0,T)$}\,,\\
\n\cdot u=0\,,&\text{in $\Omega\times(0,T)$}\,,\\
u(x,t)\cdot n(x)=0\,,&\text{on $\p\Omega\times(0,T)$}\,.
\end{cases}
\end{equation}
where $T$ is a positive constant, and $n(x)$ is the outward unit normal vector field to the boundary $\p\Omega$.

We say that $(u(x,t),p(x,t))$ is a weak solution of \eqref{eq:E} in $\Omega\times(0,T)$, if $u\in L^{\infty}(0,T;L^2(\Omega))$, $p\in L^1_{loc}(\Omega\times(0,T))$, $\n\cdot u=0$ in $\Omega\times(0,T)$, $u\cdot\ n=0$ on $\p\Omega\times(0,T)$ and, moreover,
\begin{equation}\label{Def:weak}
\langle u,\p_t\psi\rangle_x+\langle u\otimes u:\n\psi\rangle_x+\langle p,\n\cdot\psi\rangle_x=0\,,\quad\text{in $L^1(0,T)$}\,,
\end{equation}
for all vector field $\psi(x,t)\in \mathcal{D}(\Omega\times(0,T))$. We have used the notation $\langle\cdot,\cdot\rangle_x$ in \cite{BT}, which stands for the distributional duality with respect to the spatial variable $x$.\par%
Onsager's conjecture for solutions to the Euler equations may be stated as follows: Conservation of energy holds if the weak solution $u\in L^3((0,T);C^{0,\al}(\bar{\Omega}))$, with $\al>\f13$; Dissipative solutions $u\in L^3((0,T);C^{0,\al}(\bar{\Omega}))$ should exist for $\al<\f13$.
This conjecture has been intensively studied by many mathematicians for the last two decades. In the absence of a physical boundary (namely the case of whole space $\R^n$ or the case of periodic boundary conditions in the torus $\mathbb{T}^n$), Eyink in \cite{Eyink} proved that Onsager's conjecture holds if $\al>\f12$. Later, a complete proof was established by Constantin, E, and Titi in \cite{CET}, for $\al>\f13\,,$ under slightly weaker regularity assumptions on the solution. In \cite{CCFS}  Cheskidov, Constantin, Friedlander, and Shvydkoy proved energy equality in the space periodic case for solutions
$u\in L^3([0,T];B^{\f13}_{3,c(N)})$, where $B^{\f13}_{3,c(N)}$ is a Besov type space for which  $B^{\f13}_{3,p}\subset B^{\f13}_{3,c(N)}\subset  B^{\f13}_{3,\infty}$, for $1\leq p<\infty$, see reference \cite{CCFS} for details. See also the end of this section.\par%
Recently, Bardos and Titi \cite{BT} considered the Onsager's conjecture in bounded domains under the non-slip boundary condition. They proved energy conservation if $u\in L^3((0,T);C^{0,\al}(\bar{\Omega}))$, for $\al>\f13$. Later on, Bardos, Titi and Wiedemann \cite{BTW} relax this assumption, requiring only interior H\"{o}lder regularity and continuity of the normal component of the energy flux near the boundary. See also \cite{N}.  The result obtained in \cite{BTW} is particularly significant from the physical point of view. A very interesting extension of Onsager's conjecture to a class of conservation laws that possess generalized entropy is shown in by Bardos, Gwiazda, \'{S}wierczewska-Gwiazda, Titi, and Wiedemann in reference \cite{BGGTW}.

Concerning the second part of Onsager's conjecture, in a series of papers, Isett \cite{Isett}, Buckmaster, De Lellis, Szekelyhidi and Vicol \cite{BDSV}, see references therein, by using the convex integration machinery, proved the existence of dissipative energy weak solutions for any $\al<\f13$. Furthermore, Isett \cite{Isett1}  constructed energy non-conserving solutions under the assumption
$$
|u(x+y,t)-u(x,t)|\leq C|y|^{\f13-B\sqrt{\f{\log\log |y|^{-1}}{\log |y|^{-1}}}}
$$
for some constants $C$ and $B$ and for all $(x,t)$ and all $|y|\leq 10^{-2}$.

In this note we will study Onsager's conjecture in a new class of functional spaces, H\"{o}log spaces, which have been considered by the first author in \cite{BV}. To state our main result, we first introduce the definition of H\"{o}log spaces.
\begin{definition}
For each $0\leq \al<1$ and each $\la\in\R$, set
\begin{equation*}
C^{0,\al}_{\la}(\bar{\Omega})=\{f\in C(\bar{\Omega}):[f]_{C^{0,\al}_{\la}(\bar{\Omega})}<\infty\}\,,
\end{equation*}
where
\begin{equation}\label{def:holog1}
[f]_{C^{0,\al}_{\la}(\bar{\Omega})}=\sup_{x,y\in\bar{\Omega},\,,0<|x-y|<1}\f{|f(x)-f(y)|}{\left(\log\f{1}{|x-y|}\right)^{-\la}|x-y|^{\al}}\,.
\end{equation}
A norm is introduced in $C^{0,\al}_{\la}(\bar{\Omega})$ by setting $\|f\|_{C^{0,\al}_{\la}(\bar{\Omega})}\equiv[f]_{C^{0,\al}_{\la}(\bar{\Omega})}+\|f\|_{C(\bar{\Omega})}$.
\end{definition}
Now we can state our main theorem.
\begin{theorem}\label{theorem}
Assume that
\begin{equation}\label{assumption}
u\in L^3((0,T);C^{0,\al}_{\la}(\bar{\Omega}))\,,
\end{equation}
with $\al\geq\f13$ and $\la>0$. Then the weak solution of \eqref{eq:E} satisfies the energy conservation:
\begin{equation}\label{eq:EE}
\|u(\cdot,t_2)\|_{L^2(\Omega)}=\|u(\cdot,t_1)\|_{L^2(\Omega)}\,,\quad\text{for any $t_1\,,t_2\in(0,T)$\,.}
\end{equation}
\end{theorem}
Clearly, for $\al>\f13$ the above results follow immediately from the relation $C^{0,\al}_{\la}(\bar{\Omega})\subset C^{0,\al}(\bar{\Omega})$. The new results are obtained for $\,\al=\f13\,.$\par%
As still remarked in the abstract, the proof of the above result is a step by step adaptation of that in reference \cite{BT}. So we are aware that the merit of the results goes in a greater part to the above authors. However the new results are significantly stronger then the previous ones, in particular in the form stated in the following Theorem.
\begin{theorem}\label{theorem-2}
Theorem \ref{theorem} still hold if one replace in \eqref{def:holog1} the function $\,\Big(\log\f{1}{|x-y|}\Big)^{-\la}$ by $\,\omega(|x-y|)$, where $\omega(s)$ is a positive and non-decreasing function for $s> 0$, and $\lim_{s\rightarrow 0} \omega(s)=\omega(0)=0.$
\end{theorem}
The reason that led us to put in light the $C^{0,\f13}_{\la}(\bar{\Omega})$ case instead of the stronger case considered in Theorem \ref{theorem-2} is due to the effort employed by us to try to prove the first case, before realizing that the way followed in reference \cite{BT} could be applied successfully.%

\vspace{0.2cm}

Let's end this section by a comparison between the distinct results. Concerning Theorem \ref{theorem}, the gap between the set consisting of all H\"{o}lder spaces $\,C^{0,\al}(\bar{\Omega})\,,$ with $\al>\f13\,,$ and a fixed  H\"{o}log space $\,C^{0,\f13}_{\la}(\bar{\Omega})$ is wide. In fact, the union of all the above H\"{o}lder spaces is contained in the single space $\,C^{0,\f13}_{2\la}(\bar{\Omega})\,,$ which is away from $\,C^{0,\f13}_{\la}(\bar{\Omega})\,.$ Nevertheless, in comparison to the result stated in Theorem \ref{theorem-2}, also the spaces $\,C^{0,\f13}_{\la}(\bar{\Omega})\,$ are far from $\,C^{0,\f13}(\bar{\Omega})\,$. In fact, roughly speaking, we may say that there is few "free space" between the set of spaces considered in this last theorem, and $\,C^{0,\f13}(\bar{\Omega})\,.$ Recall also the sharp result, still referred above, obtained for the space periodic case in reference \cite{CCFS}. Concerning this point, let's consider the relation between $B^{\f13}_{3,c(\N)}$ and H\"{o}log spaces $\,C^{0,\f13}_{\la}(\bar{\Omega})\,$. The Besov space $B^{\f13}_{3,\infty}$ can be characterized as follows, see Proposition 8$'$ in \cite{Stein}:
\[
B^{\f13}_{3,\infty}=:\left\{f\in L^3:\|f\|_3+\sup_{|y|>0}\f{\|f(x+y)+f(x-y)-2f(x)\|_{3}}{|y|^{\f13}}<\infty\right\}\,.
\]
Hence one has $C^{0,\f13}_{\la}\subset B^{\f13}_{3,\infty}$, for any $\la>0$. From Shvydkoy \cite{S}, $c(\N)$ stands to indicate
\[
\f{1}{|y|}\int_{\mathbb{T}^n}|f(x-y)-f(x)|^3dx\to0,\quad\text{as $|y|\to0$}\,,
\]
which implies that $C^{0,\f13}_{\la}\subset B^{\f13}_{3,c(\N)}$. Hence, in the case of period domain, our $C^{0,\f13}_{\la}$ result is covered by that of Cheskidov, Constantin, Friedlander, and Shvydkoy's.
\section{Theorem \ref{theorem} for the period domain $\mathbb{T}^n$}
Before proving Theorem \ref{theorem} we consider a simpler situation, the period domain case. This helps us to understand the proof of the general bounded domain case. In this case, as in \cite{CET}, taking in \eqref{Def:weak} $\psi=(u^{\e})^{\e}$, one can get
\begin{equation*}
\f12\f{d}{dt}\int_{\mathbb{T}^n} |u^{\e}|^2dx+\int_{\mathbb{T}^n} (u\otimes u)^{\e}:\n u^{\e}dx=0\,,
\end{equation*}
which shows that
\begin{equation}\label{eq:testp}
\|u^{\e}(t_2)\|^2-\|u^{\e}(t_1)\|^2=-2\int_{t_1}^{t_2}\int_{\mathbb{T}^n} (u\otimes u)^{\e}:\n u^{\e}dx
\end{equation}
where, as \cite{CET}, we introduce a nonnegative radially symmetric $C^{\infty}(\R^n)$ mollifier, $\phi(x)$, with support in $|x|\leq1\,$ and $\int_{\R^n}\phi(x)dx=1$, and for any $0<\e<1$ we define $\phi_{\e}=\f{1}{\e^n}\phi(\f{x}{\e})$ and set $u^{\e}=u\ast\phi_{\e}$.

Now, we estimate the term on the right side in \eqref{eq:testp}. Firstly, it is well known that, see \cite{CET},
\begin{equation*}
(u\otimes u)^{\e}(x)-(u^{\e}\otimes u^{\e})(x)=\int_{\mathbb{T}^n} (\delta_yu\otimes\delta_yu)(x)\phi_{\e}(y)dy-(u-u^{\e})(x)\otimes(u-u^{\e})(x)\,,
\end{equation*}
where
\begin{equation*}
(\delta_{y})u(x)=u(x-y)-u(x)\,.
\end{equation*}
Secondly, one has, for almost all $t\in(0,T)$,
\begin{equation}\label{pf:p1-1}
|u(x-y)-u(x)|\leq \left(\log\f{1}{|y|}\right)^{-\la}|y|^{\al}\|u\|_{C^{0,\al}_{\la}}\,,\quad\text{for any $0<|y|<1$}\,,
\end{equation}
which gives
\begin{equation}\label{pf:p1}
|u(x)-u^{\e}(x)|=\left|\int_{\mathbb{T}^n}(u(x)-u(x-y))\phi_{\e}(y)dy\right|\leq \left(\log\f{1}{\e}\right)^{-\la}\e^{\al}\|u\|_{C^{0,\al}_{\la}}\,,
\end{equation}
Furthermore, one has
\begin{equation}\label{pf:p2}
\begin{split}
|\n u^{\e}(x)|=&\left|\int_{\mathbb{T}^n} \n\phi_{\e}(z)\cdot u(x-z)dz\right|\\
=&\left|\int_{\mathbb{T}^n} \n\phi_{\e}(z)\cdot (u(x-z)-u(x))dz\right|\\
\leq& C\left(\log\f{1}{\e}\right)^{-\la}\e^{\al}\|u\|_{C^{0,\al}_{\la}}\int_{\mathbb{T}^n} |\n \phi_{\e}(z)|dz\\
\leq& C\left(\log\f{1}{\e}\right)^{-\la}\e^{\al-1}\|u\|_{C^{0,\al}_{\la}}\,,
\end{split}
\end{equation}
and
\begin{equation}\label{pf:p3}
\begin{split}
\left|\int_{\mathbb{T}^n} (\delta_yu\otimes\delta_yu)(x)\phi_{\e}(y)dy\right|\leq&C\left[\left(\log\f{1}{\e}\right)^{-\la}\e^{\al}\|u\|_{C^{0,\al}_{\la}}\right]^2\int_{\mathbb{T}^n} \phi_{\e}(y)dy\,,\\
=&C\left[\left(\log\f{1}{\e}\right)^{-\la}\e^{\al}\|u\|_{C^{0,\al}_{\la}}\right]^2\,.
\end{split}
\end{equation}
Note that the estimates \eqref{pf:p1}-\eqref{pf:p2} are point-wise. In this sense they are stronger than the related estimates (6)-(8) in \cite{CET}.\par%
Finally, noting that
\begin{equation*}
\int_{\mathbb{T}^n} u^{\e}\otimes u^{\e}:\n u^{\e}dx=\int_{\mathbb{T}^n} u^{\e}\cdot\n \f12|u^{\e}|^2dx=\int_{\mathbb{T}^n} \f12|u^{\e}|^2\n\cdot u^{\e} dx=0\,,
\end{equation*}
one can deduce from \eqref{pf:p1}-\eqref{pf:p3} that
\begin{equation*}
\begin{split}
&\left|\int_{t_1}^{t_2}\int_{\mathbb{T}^n} (u\otimes u)^{\e}:\n u^{\e}dxdt\right|\\
\leq& \int_{t_1}^{t_2}\int_{\mathbb{T}^n}\left(\left|\int (\delta_yu\otimes\delta_yu)(x)\phi_{\e}(y)dy\right|+|u-u^{\e}|^2\right)|\n u^{\e}(x)|dxdt\\
\leq&C\int_{t_1}^{t_2}\left[\left(\log\f{1}{\e}\right)^{-\la}\e^{\al}\|u\|_{C^{0,\al}_{\la}}\right]^2\left(\log\f{1}{\e}\right)^{-\la}\,\e^{\al-1}\,\|u\|_{C^{0,\al}_{\la}}dt\\
=& C\,\left(\log\f{1}{\e}\right)^{-3\la}\e^{3\al-1}\int_{t_1}^{t2}\|u\|^3_{C^{0,\al}_{\la}}dt\,.
\end{split}
\end{equation*}
From this estimate, letting $\e\to0$ in \eqref{eq:testp}, we obtain the Theorem \ref{theorem}, for the periodic domain case, since $\al\geq\f13$ and $\la>0$.
\section{Preliminary Results}
When we consider a bounded domain, due to the boundary effect, one can not take $u^{\e}$ as test function. To overcome this difficulty, Bardos and Titi \cite{BT} introduced a distance function: For any $x\in\bar{\Omega}$ one defines $d(x)=\inf\limits_{y\in\p\Omega}|x-y|$, and set $\Omega_h=\{x\in\Omega: d(x)<h\}$. As in \cite{BT}, since $\p\Omega$ is a $C^2$ compact manifold, there exists $h_0(\Omega)>0$ with the following properties:
\begin{itemize}
  \item For any $x\in\overline{\Omega_{h_0}}$, the function $x\mapsto d(x)$ belongs to $C^1(\overline{\Omega_{h_0}})$;
  \item for any $x\in\overline{\Omega_{h_0}}$, there exists a unique point $\sigma(x)\in\p\Omega$ such that
  \begin{equation}\label{prop:d(x)}
  d(x)=|x-\sigma(x)|,\quad \n d(x)=-n(\sigma(x))\,.
  \end{equation}
\end{itemize}
Now, let $0\leq\eta(s)\leq1$ be a $C^{\infty}(\R)$ nondecreasing function such that $\eta(s)=0$, for $s\in(-\infty,\f12]$, and $\eta(s)=1$, for $s\in[1,\infty)$. Then $\theta_h(x)=\eta(\f{d(x)}{h})$ is a $C^1(\Omega)$ function, compactly supported in $\Omega.$ Denote by the same symbol $\theta_h$ its extension by zero outside $\Omega$. Similarly, for any $w\in L^{\infty}(\Omega)$, the extension $\theta_hw$ by zero outside $\Omega$ is well defined over all $\R^n$, and will be also denoted by $\theta_hw$.

It is natural to take $\theta_h\left(\left(\theta_hu\right)^{\e}\right)^{\e}$ as a text function. Contrarily to the period domain case, now $\n\cdot \psi\neq 0$. Hence we will need to estimate the pressure in a suitable way. Actually, due to $C^{0,\al}_{\la}(\bar{\Omega})\subset C^{0,\al}(\bar{\Omega})$, we can get the following result from Proposition 1.2 in \cite{BT}.
\begin{proposition}\label{prop:reg}
Under the assumption of Theorem \ref{theorem} the pair $(u,p)$ satisfies the following regularity properties:
\begin{equation*}
u\otimes u\in L^3((0,T);L^2(\Omega))\,,\quad p\in L^{\f32}((0,T);C^{0,\al}(\bar{\Omega}))\,,
\end{equation*}
and
\begin{equation*}
\p_tu=-\n\cdot(u\otimes u)-\n p\in L^{\f32}((0,T);H^{-1}(\Omega))\,.
\end{equation*}
Furthermore, one has
\begin{equation}\label{estimate:p}
\int_0^T\|p\|^{\f32}_{C^{0,\al}(\bar{\Omega})}dt\leq C\int_0^T\|u\|^{3}_{C^{0,\al}(\bar{\Omega})}dt\leq C\int_0^T\|u\|^{3}_{C^{0,\al}_{\la}(\bar{\Omega})}dt\,.
\end{equation}
\end{proposition}
\begin{remark}
In \cite{BT}, although the authors assume $\al>\f13$, it follows from the proof of their proposition 1.2 that the result holds for any $\al>0$, especially for $\al=\f13$.
\end{remark}
\begin{remark}
According to proposition \ref{prop:J3} below, since
\[
\int_0^T\|p\|_{L^{\infty}}\|u\|_{C^{0,\al}_{\la}}dt\leq\left(\int_0^T\|p\|_{L^{\infty}}dt\right)^{\f23}\left(\int_0^T\|u\|^3_{C^{0,\al}_{\la}}dt\right)^{\f13}\,,
\]
to obtain Theorem \ref{theorem}, we merely need to have the estimate $\int_0^T\|p\|^{\f32}_{L^{\infty}}dt\leq C\int_0^T\|u\|^{3}_{C^{0,\al}_{\la}}dt$. Hence, the estimate \eqref{estimate:p} is enough to obtain our theorem.
\end{remark}

Compared with the periodic domain case, since the test function include the function $\theta_h$, we need some estimates for $\theta$.
\begin{lemma}\label{lem}
Let $h\in(0,\min\{h_0,1\})$. For any vector field $w\in C^{0,\al}_{\la}(\bar{\Omega})$, with $w\cdot n=0$ on $\p\Omega$, there exists a constant $C$ independent of $h$ such that
\begin{equation}\label{eq:lem1}
|w(x)\cdot\n\theta_h(x)|\leq C\|w\|_{C^{0,\al}_{\la}(\bar{\Omega})}\left(\log\f{1}{h}\right)^{-\la}h^{\al-1}\,,
\end{equation}
and
\begin{equation}\label{eq:lem2}
\int_{\R^n}|w(x)\cdot\n\theta_h(x)|dx\leq C\|w\|_{C^{0,\al}_{\la}(\bar{\Omega})}\left(\log\f{1}{h}\right)^{-\la}h^{\al}\,.
\end{equation}
\end{lemma}
\begin{proof}
The proof is completely similar to that of Lemma 1.3 in \cite{BT}. For completeness, and for the readers convenience, we give here the proof. When $x\in (\Omega_h)^c$, since $\n\theta_h(x)=0$, one has $w(x)\cdot \n\theta_h(x)=0$. When $x\in\Omega_h$, it follows from \eqref{prop:d(x)} that
\begin{equation*}
\n\theta_h(x)=-\f1{h}\eta'\left(\f{d(x)}{h}\right)n(\sigma(x))\,.
\end{equation*}
Noting that $w(\sigma(x))\cdot n(\sigma(x))=0$, one can get
\begin{equation*}
\begin{split}
|w(x)\cdot\n \theta_h(x)|=&\f{1}{h}\eta'\left(\f{d(x)}{h}\right)|(w(x)-w(\sigma(x)))\cdot n(\sigma(x))|\\
\leq& \f{C}{h}\|w\|_{C^{0,\al}_{\la}}\left(\log\f{1}{|x-\sigma(x)|}\right)^{-\la}|x-\sigma(x)|^{\al}\\
\leq& C\|w\|_{C^{0,\al}_{\la}}\left(\log\f{1}{h}\right)^{-\la}h^{\al-1}\,.
\end{split}
\end{equation*}
This gives \eqref{eq:lem1}. Integrating \eqref{eq:lem1} over $\R^n$, combining with the facts that the support of $\n\theta_h$ is a subset of $\overline{\Omega_h}$, and $|\Omega_h|\leq Ch$, one obtains \eqref{eq:lem2}.
\end{proof}

\section{Proof of Theorem \ref{theorem}}
In this section, we focus on the proof of Theorem \ref{theorem}. First, we set $h\in(0,\min\{h_0,1\})$ and $\e\in(0,\f{h}{4})$. As in \cite{BT}, we take in \eqref{Def:weak} $\psi=\theta_h\left(\left(\theta_hu\right)^{\e}\right)^{\e}$ as test function. Note that, due to Proposition \ref{prop:reg}, $\psi\in W^{1,3}((0,T);H^1_0(\Omega)\,.$ So it can be used as test vector field function. So one shows that
\begin{equation}\label{eq:test}
\begin{split}
\langle u,\p_t\left(\theta_h\left(\left(\theta_hu\right)^{\e}\right)^{\e}\right)\rangle_x&+\langle u\otimes u:\n\left(\theta_h\left(\left(\theta_hu\right)^{\e}\right)^{\e}\right)\rangle_x\\
&+\langle p,\n\cdot\left(\theta_h\left(\left(\theta_hu\right)^{\e}\right)^{\e}\right)\rangle_x=0\,,\quad\text{in $L^1(0,T)$}\,.
\end{split}
\end{equation}

Next, as in \cite{BT}, we establish three propositions to estimate the three terms on the left side of \eqref{eq:test}, denoted here by $J_1\,,J_2$, and $J_3$ respectively.

For $J_1$, by arguing as in \cite{BT} Proposition 2.1, one proves the following statement.
\begin{proposition}\label{prop:J1}
For any $(t_1,t_2)\in(0,T)$, one has
\begin{equation*}
\lim_{h\to0}\int_{t_1}^{t_2}J_1dt=\f12\|u(t_2)\|^2_{L^2(\Omega)}-\f12\|u(t_1)\|^2_{L^2(\Omega)}\,.
\end{equation*}
\end{proposition}
Next, we control $J_2$.
\begin{proposition}\label{prop:J2}
The following estimate holds.
\begin{equation*}
\begin{split}
|J_2|=&|\langle u\otimes u:\n\left(\theta_h\left((\theta_h u)^{\e}\right)^{\e}\right)\rangle_x|\leq C\left(\log\f{1}{h}\right)^{-\la}h^{\al}\|u\|_{C^{0,\al}_{\la}(\Omega)}\|u\|^2_{L^{\infty}}\\
&+C\left(\log\f{1}{\e}\right)^{-\la}\e^{\al-1}\|u\|_{C^{0,\al}_{\la}}\left(\left(\log\f{1}{\e}\right)^{-\la}\e^{\al}\|u\|_{C^{0,\al}_{\la}}+\f{\e}{h}\|u\|_{L^{\infty}}\right)^2\,.
\end{split}
\end{equation*}
\end{proposition}
\begin{proof}
We first write $J_2$ as
\begin{equation*}
J_{2}=\langle u\otimes u:\n\theta_h\otimes\left((\theta_h u)^{\e}\right)^{\e}\rangle_x+\langle u\otimes u:\theta_h\n\left((\theta_h u)^{\e}\right)^{\e}\rangle_x=:J_{21}+J_{22}\,.
\end{equation*}
For $J_{21}$, by Lemma \ref{lem}, one can get
\begin{equation*}
\begin{split}
|J_{21}|=&|\langle u\otimes u:\n\theta_h\otimes\left((\theta_h u)^{\e}\right)^{\e}\rangle_x|\\
=&\left|\int_{\Omega_h}\left(u\cdot\n\theta_h\right)(u\cdot\left((\theta_h u)^{\e}\right)^{\e})dx\right|\\
\leq&C\left(\log\f{1}{h}\right)^{-\la}h^{\al}\|u\|_{C^{0,\al}_{\la}}\|u\|^2_{L^{\infty}}\,.
\end{split}
\end{equation*}
For $J_{22}$, one has
\begin{equation*}
\begin{split}
J_{22}=&\langle u\otimes u:\theta_h\n\left((\theta_h u)^{\e}\right)^{\e}\rangle_x|\\
=&|\langle u\otimes \theta_hu:\n\left((\theta_h u)^{\e}\right)^{\e}\rangle_x|\\
=&\langle \left(u\otimes \theta_hu\right)^{\e}:\n(\theta_h u)^{\e}\rangle_x\\
=&\langle \left(\left(u\otimes \theta_hu\right)^{\e}-\left(u^{\e}\otimes(\theta_hu)^{\e}\right)\right):\n(\theta_h u)^{\e}\rangle_x\,,
\end{split}
\end{equation*}
where we have used that
\begin{equation*}
\int_{\R^n_x} u^{\e}\otimes (\theta_hu)^{\e}:\n (\theta_hu)^{\e}dx=\int u^{\e}\cdot\n \f12|(\theta_hu)^{\e}|^2dx=\int \f12|(\theta_hu)^{\e}|^2\n\cdot u^{\e} dx=0\,.
\end{equation*}
By using the identity
\begin{equation*}
(v\otimes w)^{\e}(x)-(v^{\e}\otimes w^{\e})(x)=\int_{\R^n_y}(\delta_{y}v\otimes\delta_yw)(x)\phi_{\e}(y)dy-(v-v^{\e})(x)\otimes(w-w^{\e})(x)\,,
\end{equation*}
where
\begin{equation*}
(\delta_{y})v(x)=v(x-y)-v(x)\,,\quad (\delta_{y})w(x)=w(x-y)-w(x)\,,
\end{equation*}
one can write $J_{22}=J_{221}+J_{222}$ with
\begin{equation*}
\begin{split}
J_{221}=&\int_{\R^n_x}\left(\int_{\R^n_y}(\delta_{y}u\otimes\delta_y(\theta_hu))(x)\phi_{\e}(y)dy\right):\left(\int_{\R^n_z}\n\phi_{\e}(z)\otimes(\theta_hu)(x-z)dz\right)dx\\
=&\int_{\Omega}\left(\int_{\R^n_y}(\delta_{y}u\otimes\delta_y(\theta_hu))(x)\phi_{\e}(y)dy\right):\left(\int_{\R^n_z}\n\phi_{\e}(z)\otimes(\theta_hu)(x-z)dz\right)dx\,.
\end{split}
\end{equation*}
and
\begin{equation*}
\begin{split}
J_{222}=&\int_{\R^n_x}\left((u-u^{\e})\otimes((\theta_hu)-(\theta_hu)^{\e})\right):\n(\theta_hu)^{\e}dx\\
=&\int_{\Omega}\left((u-u^{\e})\otimes((\theta_hu)-(\theta_hu)^{\e})\right):\n(\theta_hu)^{\e}dx\,.
\end{split}
\end{equation*}
For $J_{221}$, noting that supp $\phi_{\e}\subset\{y:|y|\leq\e\}$, that $|\delta_y\theta_h(x)|\leq C\f{\e}{h}$ for all $|y|\leq\e$, and that $\int_{\R^n_z}|\n\phi_{\e}|dz\leq C\e^{-1}$, one shows that
\begin{equation*}
\begin{split}
&\left|\int_{\R^n_y}(\delta_{y}u\otimes\delta_y(\theta_hu))(x)\phi_{\e}(y)dy\right|\\
&=\left|\int_{\R^n_y}(\delta_{y}u\otimes\left(\theta_h(x-y)(\delta_yu)(x)+(\delta_y\theta_h)(x)u(x-y)\right)\phi_{\e}(y)dy\right|\\
&\leq C\left(\log\f{1}{\e}\right)^{-\la}\e^{\al}\|u\|_{C^{0,\al}_{\la}}\int_{\R^n_y}\left(\left(\log\f{1}{\e}\right)^{-\la}\e^{\al}\|u\|_{C^{0,\al}_{\la}}+\f{\e}{h}\|u\|_{L^{\infty}}\right)\phi_{\e}(y)dy\\
&=C\left(\log\f{1}{\e}\right)^{-\la}\e^{\al}\|u\|_{C^{0,\al}_{\la}}\left(\left(\log\f{1}{\e}\right)^{-\la}\e^{\al}\|u\|_{C^{0,\al}_{\la}}+\f{\e}{h}\|u\|_{L^{\infty}}\right)\,,
\end{split}
\end{equation*}
and
\begin{equation}\label{eq:J201}
\begin{split}
&\left|\int_{\R^n_z}\n\phi_{\e}(z)\otimes(\theta_hu)(x-z)dz\right|\\
=&\left|\int_{\R^n_z}\n\phi_{\e}(z)\otimes\left((\theta_hu)(x-z)-(\theta_hu)(x)\right)dz\right|\\
=&\left|\int_{\R^n_z}\n\phi_{\e}(z)\otimes\left(\delta_z\theta_h(x)u(x-z)-\theta_h(x)\delta_zu(x)\right)dz\right|\\
\leq&C\left(\left(\log\f{1}{\e}\right)^{-\la}\e^{\al}\|u\|_{C^{0,\al}_{\la}}+\f{\e}{h}\|u\|_{L^{\infty}}\right)\int_{\R^n_z}|\n\phi_{\e}(z)|dz\\
\leq& C\e^{-1}\left(\left(\log\f{1}{\e}\right)^{-\la}\e^{\al}\|u\|_{C^{0,\al}_{\la}}+\f{\e}{h}\|u\|_{L^{\infty}}\right)\,.
\end{split}
\end{equation}
Hence, one has
\begin{equation*}
|J_{221}|\leq C\left(\log\f{1}{\e}\right)^{-\la}\e^{\al-1}\|u\|_{C^{0,\al}_{\la}}\left(\left(\log\f{1}{\e}\right)^{-\la}\e^{\al}\|u\|_{C^{0,\al}_{\la}}+\f{\e}{h}\|u\|_{L^{\infty}}\right)^2\,.
\end{equation*}
For $J_{222}$, it follows from \eqref{eq:J201} that
\begin{equation}\label{eq:J202}
\begin{split}
|\n(\theta_hu)^{\e}(x)|=&\left|\int_{\R^n_z} \n\phi_{\e}(z)\otimes(\theta_hu)(x-z)dz\right|\\
\leq& C\e^{-1}\left(\left(\log\f{1}{\e}\right)^{-\la}\e^{\al}\|u\|_{C^{0,\al}_{\la}}+\f{\e}{h}\|u\|_{L^{\infty}}\right)\,.
\end{split}
\end{equation}
On the other hand, for all $x\in\text{supp}\,\theta_{h+\e}$, one has
\begin{equation}\label{eq:J203}
|u(x)-u^{\e}(x)|=\left|\int_{\R^n_y}(u(x)-u(x-y))\phi_{\e}(y)dy\right|\leq \left(\log\f{1}{\e}\right)^{-\la}\e^{\al}\|u\|_{C^{0,\al}_{\la}}\,,
\end{equation}
and
\begin{equation}\label{eq:J204}
\begin{split}
&|(\theta_hu)(x)-(\theta_hu)^{\e}(x)|\\
=&\left|\int_{\R^n_y}((\theta_hu)(x)-(\theta_hu)(x-y))\phi_{\e}(y)dy\right|\\
=&\left|\int_{\R^n_y}\phi_{\e}(y)\left(\delta_y\theta_h(x)u(x-y)-\theta_h(x)\delta_yu(x)\right)dy\right|\\
\leq&C\left(\left(\log\f{1}{\e}\right)^{-\la}\e^{\al}\|u\|_{C^{0,\al}_{\la}}+\f{\e}{h}\|u\|_{L^{\infty}}\right)\,.
\end{split}
\end{equation}
Combining with \eqref{eq:J202}-\eqref{eq:J204}, one gets
\begin{equation*}
|J_{222}|\leq C\left(\log\f{1}{\e}\right)^{-\la}\e^{\al-1}\|u\|_{C^{0,\al}_{\la}}\left(\left(\log\f{1}{\e}\right)^{-\la}\e^{\al}\|u\|_{C^{0,\al}_{\la}}+\f{\e}{h}\|u\|_{L^{\infty}}\right)^2\,.
\end{equation*}
Now, collecting the above estimates obtained for $J_{21}$, $J_{221}$, and $J_{222}$, one obtains the desired estimate for $J_2$.
\end{proof}
Finally, we estimate $J_3$.
\begin{proposition}\label{prop:J3}
One has
\begin{equation*}
|\langle p, \n\cdot\left(\theta_h\left((\theta_hu)^{\e}\right)^{\e}\right)\rangle_x|\leq C\|p\|_{L^{\infty}}\|u\|_{C^{0,\al}_{\la}}\left(\left(\log\left(\f{1}{h}\right)\right)^{-\la}h^{\al}+\left(\log\left(\f{1}{\e}\right)\right)^{-\la}\e^{\al}\right)\,.
\end{equation*}
\end{proposition}
\begin{proof}
First, one has
\begin{equation*}
\begin{split}
\langle p, \n\cdot\left(\theta_h\left((\theta_hu)^{\e}\right)^{\e}\right)\rangle_x=&\int_{\Omega}p\n\cdot\left(\theta_h\left((\theta_hu)^{\e}\right)^{\e}\right)dx\\
=&\int_{\Omega}(p\theta_h)\n\cdot\left((\theta_hu)^{\e}\right)^{\e}dx+\int_{\Omega}p\n\theta_h\cdot\left((\theta_hu)^{\e}\right)^{\e}dx\\
=&:J_{31}+J_{32}\,.
\end{split}
\end{equation*}
Concerning $J_{31}$, from (2.22) and (2.25) in Proposition 2.3 of \cite{BT}, one obtains
\begin{equation*}
J_{31}=\int_{\Omega}\Big(p(x)\theta_h(x)\int_{\R^n_y}\int_{\R^n_z}\phi_{\e}(x-y)\phi_{\e}(z-y)u(z)\cdot\n\theta_h(z)dzdy\Big)dx\,,
\end{equation*}
by Lemma \ref{lem}, which implies that
\begin{equation*}
|J_{31}|\leq C\|p\|_{L^{\infty}}\|u\|_{C^{0,\al}_{\la}}\left(\log\left(\f{1}{h}\right)\right)^{-\la}h^{\al}\,.
\end{equation*}
For $J_{32}$, as in \cite{BT}, one has
\begin{equation*}
\begin{split}
J_{32}=&\int_{\Omega_h}\left(p(x)\n\theta_h(x)\cdot\int_{\R^n_z}\int_{\R^n_y}\theta_h(x-y+z)u(x-y+z)\phi_{\e}(y)\phi_{\e}(z)dydz\right)dx\,,\\
=&\int_{\Omega_h}p(x)\left(\int_{\R^n_z}\int_{\R^n_y}\phi_{\e}(y)\phi_{\e}(z)\theta_h(x-y+z)\left(u(x-y+z)-u(x)\right)\cdot\n\theta_h(x)dydz\right)dx\\
&+\int_{\Omega_h}p(x)\left(\int_{\R^n_z}\int_{\R^n_y}\phi_{\e}(y)\phi_{\e}(z)\theta_h(x-y+z)u(x)\cdot\n\theta_h(x)dydz\right)dx\\
=&:J_{321}+J_{322}\,.
\end{split}
\end{equation*}
Noting that
$$
|u(x-y+z)-u(x)|\leq C\|u\|_{C^{0,\al}_{\la}}\left(\log\left(\f{1}{\e}\right)\right)^{-\la}\e^{\al}\,,
$$
for the relevant $x,\,y,\,z$ for which the integrand in the definition of $J_{321}$ is not zero, and that $\int_{\Omega_h}|\n\theta_h(x)|dx\leq C\,,$ one shows that
\begin{equation*}
|J_{321}|\leq C\|p\|_{\infty}\|u\|_{C^{0,\al}_{\la}}\left(\log\left(\f{1}{\e}\right)\right)^{-\la}\e^{\al}\,.
\end{equation*}
Concerning $J_{322}$, it follows from Lemma \ref{lem} that
\begin{equation*}
\begin{split}
|J_{322}|\leq& \int_{\Omega_h}|p(x)|\left(\int_{\R^n_z}\int_{\R^n_y}\phi_{\e}(y)\phi_{\e}(z)|u(x)\cdot\n\theta_h(x)|dydz\right)dx\\
\leq&  C\|p\|_{\infty}\|u\|_{C^{0,\al}_{\la}}\left(\log\left(\f{1}{h}\right)\right)^{-\la}h^{\al}\,.
\end{split}
\end{equation*}
Collecting the above estimates, one proves the proposition.
\end{proof}
Now, it follows from Propositions \ref{prop:J2}, \ref{prop:J3} and the estimate \eqref{estimate:p} in Proposition \ref{prop:reg} that
\begin{equation}\label{eq:proof}
\begin{split}
\int_{t_1}^{t_2}|J_2+J_3|dt\leq C\left(\left(\log\left(\f{1}{h}\right)\right)^{-\la}h^{\al}+\left(\log\left(\f{1}{\e}\right)\right)^{-\la}\e^{\al}
\right)\int_{t_1}^{t_2}\|u\|^3_{C^{0,\al}_{\la}}dt\\
+\left(\left(\log\left(\f{1}{\e}\right)\right)^{-3\la}\e^{3\al-1}+
\left(\log\left(\f{1}{\e}\right)\right)^{-\la}\f{\e^{\al+1}}{h^2}\right)\int_{t_1}^{t_2}\|u\|^3_{C^{0,\al}_{\la}}dt\,.
\end{split}
\end{equation}
by choosing $\e=h^{\f{2}{1+\al}}$ and by letting $h\to0$, since $\al\geq\f13$ and $\la>0$, one has $\int_{t_1}^{t_2}|J_2+J_3|dt\to 0$. Combining this fact with Proposition \ref{prop:J1} and equation \eqref{eq:test}, one proves the Theorem \ref{theorem}.


\begin{thebibliography}{1}
\bibitem{BT}
Bardos, C., Titi, E. S.: Onsager's conjecture for the incompressible Euler equations in bounded domains, \textit{Archive for Rational Mechanics and Analysis}, \textbf{228}(1), 197--207 2018
\bibitem{BGGTW}
Bardos, C., Gwiazda, P., \'{S}wierczewska-Gwiazda, A., Titi, E. S., Wiedemann, E.: On the extension of Onsager's conjecture for general conservation laws, \textit{Journal of Nonlinear Science}, \textbf{29}(2), 501--510 2019
\bibitem{BTW}
Bardos, C., Titi, E. S., Wiedemann, E.: Onsager's conjecture with physical boundaries and an application to the vanishing viscosity limit, \textit{Communications in Mathematical Physics}, \textbf{370}(1), 291--310 2019
\bibitem{BV}
Beir\~{a}o da Veiga, H.: Moduli of continuity, functional spaces,\break and elliptic boundary value problems, The full regularity spaces $C^{0,\la}_{\al}$, \textit{Advances in Nonlinear Analysis}, \textbf{7}(1), 15--34 2018
\bibitem{BDSV}
Buckmaster, T., De Lellis, C., Sz\'{e}kelyhidi, L.Jr., Vicol, V.: Onsager's conjecture for admissible weak solutions, arXiv:1701.08678.
\bibitem{CCFS}
Cheskidov, A.,  Constantin, P., Friedlander, S., Shvydkoy, R.: Energy conservation and Onsager's conjecture for the Euler equations, \textit{Nonlinearity}, \textbf{21}(6),1233--1252 2008
\bibitem{CET}
Constantin, P.,W. E., Titi, E.S.: Onsager's conjecture on the energy conservation for solutions of Euler's equation, \textit{Commun. Math. Phys.}, \textbf{165}, 207--209 1994
\bibitem{Eyink}
Eyink, G.L.: Energy dissipation without viscosity in ideal hydrodynamics, I. Fourier analysis and local energy transfer, \textit{Phys. D}, \textbf{78}(3-4), 222--240 1994
\bibitem{Isett}
Isett, P.: A proof of Onsager's conjecture, \textit{Annals of Mathematics}, \textbf{188}(3), 871--963 2018
\bibitem{Isett1}
Isett, P.: On the endpoint regularity in Onsager's conjecture, arXiv:1706.01549.
\bibitem{N}
Nguyen, Q. H., Nguyen, P. T.: Onsager's conjecture on the energy conservation for solutions of Euler equations in bounded domains, \textit{Journal of Nonlinear Science}, \textbf{29}(1), 207-213 2019
\bibitem{Onsager}
Onsager, L.: Statistical hydrodynamics, \textit{Nuovo Cimento}, \textbf{(9) 6}, Supplemento, 2(ConvegnoInternazionale di Meccanica Statistica), 279--287 1949.
\bibitem{S}
Shvydkoy R.: On the energy of inviscid singular flows, \textit{J. Math. Anal. Appl.}, \textbf{349}, 583--595 2009
\bibitem{Stein}
 Stein, E. M.: Singular Integrals and Differentiability Properties of Functions (Vol. 2). Princeton University Press, 1970.
\end{thebibliography}
\end{document}